\documentclass[12pt]{article}
\usepackage{amsfonts}
\usepackage{amssymb}
\usepackage{amsmath}

\newtheorem{theorem}{Theorem}[section]

\newtheorem{corollary}{Corollary}[section]

\newtheorem{example}{Example}[section]

\newtheorem{lemma}{Lemma}[section]

\newtheorem{remark}{Remark}[section]

\newenvironment{proof}[1][Proof]{\noindent\textbf{#1.} }{\ \rule{0.5em}{0.5em}}
\input{tcilatex}
\begin{document}

\title{\v{S}migoc's glue for universal realizability in the left half-plane%
\thanks{%
Supported by Universidad Cat\'{o}lica del Norte-VRIDT 036-2020, N\'{u}cleo 6
UCN VRIDT 083-2020, Chile.}}
\author{Jaime H. Alfaro, Ricardo L. Soto\thanks{%
E-mail addresses: rsoto@ucn.cl (R. L. Soto), jaime.alfaro@ucn.cl (J. H.
Alfaro).} \\
{\small Dpto. Matem\'{a}ticas, Universidad Cat\'{o}lica del Norte, Casilla
1280}\\
{\small Antofagasta, Chile.}}
\date{}
\maketitle

\begin{abstract}
A list $\Lambda =\{\lambda _{1},\lambda _{2},\ldots ,\lambda _{n}\}$ of
complex numbers is said to be \textit{realizable} if it is the spectrum of a
nonnegative matrix. $\Lambda $ is said to be \textit{universally realizable}
($\mathcal{UR}$) if it is realizable for each possible Jordan canonical form
allowed by $\Lambda .$ In this paper, using companion matrices and applying
a procedure by \v{S}migoc, is provides a sufficient condition for the 
\textit{universal realizability} of left half-plane spectra, that is, $%
\Lambda =\{\lambda _{1},\ldots ,\lambda _{n}\}$ with $\lambda _{1}>0,$ $%
\func{Re}\lambda _{i}\leq 0,$ $i=2,\ldots ,n.$ It is also shown how the
effect of adding a negative real number to a not $\mathcal{UR}$ left
half-plane list of complex numbers, makes the new list $\mathcal{UR}$, and a
family of left half-plane lists that are $\mathcal{UR}$ is characterized.
\end{abstract}



\section{Introduction}

\noindent A list $\Lambda =\{\lambda _{1},\lambda _{2},\ldots ,\lambda
_{n}\} $ of complex numbers is said to be \textit{realizable} if it is the
spectrum of an $n$-by-$n$ nonnegative matrix $A$, and $A$ is said to be a 
\textit{realizing} \textit{matrix} for\textit{\ }$\Lambda .$ The problem of
the realizability of spectra is called the \textit{nonnegative inverse
eigenvalue problem }(NIEP). From the Perron-Frobenius Theorem we know that
if $\Lambda =\{\lambda _{1},\lambda _{2},\ldots ,\lambda _{n}\}$ is the
spectrum of an $n$-by-$n$ nonnegative matrix $A,$ then the leading
eigenvalue of $A$ equals to the spectral radius of $A,$ $\rho
(A)=:\max\limits_{1\leq i\leq n}\left\vert \lambda _{i}\right\vert .$ This
eigenvalue is called the \textit{Perron eigenvalue,} and we shall assume in
this paper, that $\rho (A)=\lambda _{1}.$

\bigskip

A matrix is said to have \textit{constant row sums, }if each one of its rows
sums up to the same constant $\alpha .$ The set of all matrices with
constant row sums equal to $\alpha ,$ is denoted by $\mathcal{CS}_{\alpha }.$%
\emph{\ }Then, any matrix $A\in \mathcal{CS}_{\alpha }$ has the eigenvector $%
\mathbf{e}^{T}=[1,1,\ldots ,1],$ corresponding to the eigenvalue $\alpha .$
The real matrices with constant row sums are important because it is known
that the problem of finding a nonnegative matrix with spectrum $\Lambda
=\{\lambda _{1},\ldots ,\lambda _{n}\},$ is equivalent to the problem of
finding a nonnegative matrix in \emph{\ }$\mathcal{CS}_{\lambda _{1}}$ with
spectrum $\Lambda $ (see \cite{Johnson1}). We denote by $\mathbf{e}_{k},$
the n-dimensional vector, with $1$ in the $k^{th}$ position and zeros
elsewhere. If $\Lambda =\{\lambda _{1},\ldots ,\lambda _{n}\},$ then $%
s_{k}(\Lambda )=\dsum\limits_{i=1}^{n}\lambda _{i}^{k},$ $k=1,2,\ldots .$

\bigskip

A list $\Lambda =\{\lambda _{1},\lambda _{2},\ldots ,\lambda _{n}\}$ of
complex numbers, is said to be \textit{diagonalizably realizable }($\mathcal{%
DR}$)\textit{,} if there is a diagonalizable realizing matrix for $\Lambda $
\ The list $\Lambda $ is said to be \textit{universally realizable }($%
\mathcal{UR}$), if it is realizable for each possible Jordan canonical form
(JCF) allowed by $\Lambda .$ The problem of the universal realizability of
spectra, is called the \textit{universal realizability problem }(URP). The
URP contains the NIEP, and both problems are equivalent if the given numbers 
$\lambda _{1},\lambda _{2},\ldots ,\lambda _{n}$ are distinct. In terms of $%
n,$ both problems remain unsolved for $n\geq 5.$ It is clear that if $%
\Lambda $ is $\mathcal{UR}$, then $\Lambda $ must be $\mathcal{DR}$. The
first known results on the URP are due to Minc \cite{Minc1, Minc2}. In terms
of the URP, Minc \cite{Minc1} showed that if a list $\Lambda =\{\lambda
_{1},\lambda _{2},\ldots ,\lambda _{n}\}$ of complex numbers is the spectrum
of a diagonalizable positive matrix, then $\Lambda $ is $\mathcal{UR}$. The
positivity condition is necessary for Minc's proof, and the question set by
Minc himself, whether the result holds for nonnegative realizations was open
for almost 40 years. Recently, two extensions of Minc's result have been
obtained in \cite{Collao, Johnson4}. In \cite{Collao}, Collao et al. showed
that a nonnegative matrix $A\in \mathcal{CS}_{\lambda _{1}},$ with a
positive column, is similar to a positive matrix. Note that if $A$ is
nonnegative with a positive row and $A^{T}$ has a positive eigenvector, then 
$A^{T}$ is also similar to a positive matrix. Besides, if $\Lambda $ is
diagonalizably realizable by a matrix $A\in \mathcal{CS}_{\lambda _{1}}$
having a positive column, then $\Lambda $ is $\mathcal{UR}$. In \cite%
{Johnson4}, Johnson et al. introduced the concept of ODP matrices, that is,
nonnegative matrices with all positive off-diagonal entries (zero diagonal
entries are permitted) and proved that if $\Lambda $ is diagonalizably ODP
realizable, then $\Lambda $ is $\mathcal{UR}$. Note that both extensions
contain, as a particular case, Minc's result in \cite{Minc1}. Both
extensions allow us to significantly increase the set of spectra that can be
proved to be $\mathcal{UR}$, as for instance, certain spectra $\Lambda
=\{\lambda _{1},\ldots ,\lambda _{n}\}$ with $s_{1}(\Lambda )=0,$ which is
not possible from Minc's result. In particular, we shall use the extension
in \cite{Collao} to generate some of our results.

\begin{remark}
In \cite{Collao}, Section $2,$ Theorem $2.1$ and Corollary $2.1,$ there is
an error in assuming that if $A$ is nonnegative with a positive row, then $%
A^{T},$ which has a positive column, is similar to a positive matrix. The
reason is that we cannot guarantee that $A^{T}$ has a positive eigenvector.
\end{remark}

Regarding non-positive universal realizations, we mention that in \cite%
{Soto, Diaz} the authors proved, respectively, that lists of complex numbers 
$\Lambda =\{\lambda _{1},\ldots ,\lambda _{n}\},$ of Suleimanova type, that
is,%
\begin{equation*}
\lambda _{1}>0,\text{ }\func{Re}\lambda _{i}\leq 0,\text{ }\left\vert \func{%
Re}\lambda _{i}\right\vert \geq \left\vert \func{Im}\lambda _{i}\right\vert ,%
\text{ }i=2,3,\ldots ,n\text{,}
\end{equation*}%
or of \v{S}migoc type, that is, 
\begin{equation}
\lambda _{1}>0,\text{ }\func{Re}\lambda _{i}\leq 0,\text{ }\sqrt{3}%
\left\vert \func{Re}\lambda _{i}\right\vert \geq \left\vert \func{Im}\lambda
_{i}\right\vert ,\text{ }i=2,3,\ldots ,n\text{,}  \label{Smi}
\end{equation}%
are $\mathcal{UR}$ if and only if they are realizable if and only if $%
\dsum\limits_{i=1}^{n}\lambda _{i}\geq 0.$

Outline of the paper:\textbf{\ }The paper is organized as follows: In
Section $2,$ we present the mathematical tools that will be used to generate
our results. In Section $3,$ we study the URP for a left half-plane list and
we give a sufficient condition for it to be $\mathcal{UR}$. In Section $4,$
we discuss the effect of adding a negative real number $-c$ to a left
half-plane list $\Lambda =\{\lambda _{1},-a\pm bi,\ldots ,-a\pm bi\}$, which
is not $\mathcal{UR}$ (or even not realizable), or we do not know whether it
is, and we show how $\Lambda \cup \{-c\}$ becomes $\mathcal{UR}$. We also
characterize a family of left half-plane lists that are $\mathcal{UR}$. In
Section $5,$ we show that the merge of two lists diagonalizably realizable $%
\Gamma _{1}\in \emph{CS}_{\lambda _{1}}$ and $\Gamma _{2}\in \emph{CS}_{\mu
_{1}}$ is $\mathcal{UR}$. Examples are shown to illustrate the results.

\section{Preliminaries}

\noindent Throughout this paper we use the following results: The first one,
by \v{S}migoc \cite{Smigoc}, gives a procedure that we call \v{S}migoc's
glue technique, to obtain from two matrices $A$ and $B$ of size $n$-by-$n$
and $m$-by-$m,$ respectively, a new $(n+m-1)$-by-$(n+m-1)$ matrix $C,$
preserving in certain way, the corresponding JCFs of $A$ and $B.$ The second
one, by Laffey and \v{S}migoc \cite{Laffey} solves the NIEP for lists of
complex numbers on the left half-plane, that is, lists with $\lambda _{1}>0,$
$\func{Re}\lambda _{i}\leq 0,$ $i=2,\ldots ,n.$ Moreover, we also use Lemma $%
5$ in \cite{Laffey}.

\begin{theorem}
\label{Smigoc} \textrm{\cite{Smigoc}} Suppose $B$ is an $m$-by-$m$ matrix
with a JCF that contains at least one $1$-by-$1$ Jordan block corresponding
to the eigenvalue $c$: 
\begin{equation*}
J(B)=\left[ 
\begin{array}{cc}
c & 0 \\ 
0 & I(B)%
\end{array}%
\right] .
\end{equation*}%
Let $\mathbf{t}$ and $\mathbf{s}$, respectively, be the left and the right
eigenvectors of $B$ associated with the $1$-by-$1$ Jordan block in the above
canonical form. Furthermore, we normalize vectors $\mathbf{t}$ and $\mathbf{s%
}$ so that $\mathbf{t}^{\textsuperscript{T}}\mathbf{s}=1.$ Let $J(A)$ be a
JCF for the $n$-by-$n$ matrix 
\begin{equation*}
A=\left[ 
\begin{array}{cc}
A_{1} & \mathbf{a} \\ 
\mathbf{b}^{T} & c%
\end{array}%
\right] ,
\end{equation*}%
where $A_{1}$ is an $(n-1)$-by-$(n-1)$ matrix and $\mathbf{a}$ and $\mathbf{b%
}$ are vectors in $\mathbb{C}^{\textsuperscript{n-1}}.$ Then the matrix 
\begin{equation*}
C=\left[ 
\begin{array}{cc}
A_{1} & \mathbf{at}^{\textsuperscript{T}} \\ 
\mathbf{sb}^{\textsuperscript{T}} & B%
\end{array}%
\right]
\end{equation*}%
has JCF 
\begin{equation*}
J(C)=\left[ 
\begin{array}{cc}
J(A) & 0 \\ 
0 & I(B)%
\end{array}%
\right] .
\end{equation*}
\end{theorem}

\begin{theorem}
\cite{Laffey} \label{LS}Let $\Lambda =\{\lambda _{1},\lambda _{2},\ldots
,\lambda _{n}\}$ be a list of complex numbers with $\lambda _{1}\geq
\left\vert \lambda _{i}\right\vert $ and $\func{Re}\lambda _{i}\leq 0,$ $%
i=2,\ldots ,n.$ Then $\Lambda $ is realizable if and only if 
\begin{equation*}
s_{1}=s_{1}(\Lambda )\geq 0,\text{ \ }s_{2}=s_{2}(\Lambda )\geq 0,\text{ \ }%
s_{1}^{2}\leq ns_{2}.
\end{equation*}
\end{theorem}

\begin{lemma}
\cite{Laffey} \label{LS2} Let $t$ be a nonnegative real number and let $%
\lambda _{2},\lambda _{3},\ldots ,\lambda _{n}$ be complex numbers with real
parts less than or equal to zero, such that the list $\{\lambda _{2},\lambda
_{3},\ldots ,\lambda _{n}\}$ is closed under complex conjugation. Set $\rho
=2t-\lambda _{2}-\cdots -\lambda _{n}$ and%
\begin{equation}
f(x)=(x-\rho )\dprod\limits_{j=2}^{n}(x-\lambda
_{j})=x^{n}-2tx^{n-1}+b_{2}x^{n-2}+\cdots +b_{n}.  \label{b2}
\end{equation}%
Then $b_{2}\leq 0$ implies $b_{j}\leq 0$ for $j=3,4,\ldots ,n.$
\end{lemma}

\section{Companion matrices and the \v{S}migoc's glue.}

\noindent We say that a list $\Lambda =\{\lambda _{1},\lambda _{2},\ldots
,\lambda _{n}\}$ of complex numbers is on the left half-plane if $\lambda
_{1}>0,$ $\func{Re}\lambda _{i}\leq 0,$ $i=2,3,\ldots ,n.$ In this section
we give a sufficient condition for a left half-plane list of complex numbers
to be $\mathcal{UR}$. Of course, it is our interest to consider lists of
complex numbers containing elements out of realizability region of lists of 
\v{S}migoc type. Our strategy consists in to decompose the given list $%
\Lambda =\{\lambda _{1},\lambda _{2},\ldots ,\lambda _{n}\}$ into sub-lists%
\begin{equation*}
\Lambda _{k}=\{\lambda _{k1},\lambda _{k2},\ldots ,\lambda _{kp_{k}}\},\text{
}\lambda _{11}=\lambda _{1},\text{ }k=1,2,\ldots ,t,
\end{equation*}%
with auxiliary lists 
\begin{eqnarray*}
\Gamma _{1} &=&\Lambda _{1} \\
\Gamma _{k} &=&\{s_{1}(\Gamma _{k-1}),\lambda _{k1},\lambda _{k2},\ldots
,\lambda _{kp_{k}}\},\text{ \ }k=2,,\ldots ,t,
\end{eqnarray*}%
each one of them being the spectrum of a nonnegative companion matrix $%
A_{k}, $ in such a way that it be possible to apply \v{S}migoc's glue
technique to the matrices $A_{k},$ to obtain an $n$-by-$n$ nonnegative
matrix with spectrum $\Lambda $ for each possible JCF allowed by $\Lambda .$
In the case $s_{1}(\Lambda )>0,$ with $\lambda _{i}\neq 0,$ $i=2,\ldots ,n,$
we may choose, if they exist, sub-lists $\Gamma _{k}$ being the spectrum of
a diagonalizable nonnegative companion matrix with a positive column. Then,
after \v{S}migoc's glue, we obtain a diagonalizable nonnegative $n$-by-$n$
matrix $A$ with spectrum $\Lambda $ and a positive column, which is similar
to a diagonalizable positive matrix. Thus, from the extension in \cite%
{Collao}, $\Lambda $ is $\mathcal{UR}$. Next we have the following corollary
from Theorem \ref{Smigoc}:

\begin{corollary}
\label{Main}Let $\Lambda =\{\lambda _{1},\lambda _{2},\ldots ,\lambda _{n}\}$
be a realizable left half-plane list of complex numbers. Suppose that for
each JCF $\mathbf{J}$ allowed by $\Lambda ,$ there exists a decomposition of 
$\Lambda $ as%
\begin{eqnarray*}
\Lambda &=&\Lambda _{1}\cup \Lambda _{2}\cup \cdots \cup \Lambda _{t},\text{
where} \\
\Lambda _{k} &=&\{\lambda _{k1},\lambda _{k2},\ldots ,\lambda _{kp_{k}}\},%
\text{ }k=1,2,\ldots ,t,\text{ }\lambda _{11}=\lambda _{1},
\end{eqnarray*}%
with auxiliary lists%
\begin{eqnarray*}
\Gamma _{1} &=&\Lambda _{1}, \\
\Gamma _{k} &=&\{s_{1}(\Gamma _{k-1}),\lambda _{k1},\lambda _{k2}\ldots
,\lambda _{kp_{k}}\},\text{ }k=2,\ldots ,t,
\end{eqnarray*}%
being the spectrum of a nonnegative companion matrix $A_{k}$ with JCF $%
J(A_{k})$ as a sub-matrix of $\mathbf{J,}$ $k=1,2,\ldots ,t$.\newline
Then $\Lambda $ is universally realizable.
\end{corollary}

\begin{proof}
Since each matrix $A_{k},$ $k=1,2,\ldots ,t$, is nonnegative companion with
JCF $J(A_{k})$ being a submatrix of $\mathbf{J,}$ then, from \v{S}migoc's
glue applied to matrices $A_{k},$ we obtain an $n$-by-$n$ nonnegative matrix
with spectrum $\Lambda $ and JCF $\mathbf{J.}$ As $\mathbf{J}$ is any JCF
allowed by $\Lambda ,$ then $\Lambda $ is $\mathcal{UR}.$
\end{proof}

\bigskip

The following result is well known and useful.

\begin{lemma}
\label{CS}Let $A$ be a diagonalizable irreducible nonnegative matrix with
spectrum $\Lambda =\{\lambda _{1},\ldots ,\lambda _{n}\}$ and a positive row
or column. Then $A$ is similar to a diagonalizable nonnegative matrix $B\in 
\emph{CS}_{\lambda _{1}},$ with a positive row or column.
\end{lemma}

\begin{proof}
If $A$ is irreducible nonnegative, it has a positive eigenvector $\mathbf{x}%
^{T}=[x_{1},\ldots ,x_{n}].$ Then if $D=dig\{x_{1},\ldots ,x_{n}\},$ the
matrix 
\begin{equation*}
{\small B=D}^{-1}{\small AD=}\left[ \frac{x_{j}}{x_{i}}a_{i,j}\right] 
{\small \in CS}_{\lambda _{1}}
\end{equation*}%
is nonnegative with a positive row or column.
\end{proof}

\bigskip

Suppose all lists $\Gamma _{k}$ in Corollary \ref{Main}, can be taken as the
spectrum of a diagonalizable nonnegative companion matrix $A_{k}$ with a
positive column (the last one). Then, since the glue of matrices $A_{k}$
gives an $n$-by-$n$ diagonalizable irreducible nonnegative matrix $A$ with a
positive column and spectrum $\Lambda $, $A$ is similar to a diagonalizable
positive matrix with spectrum $\Lambda $ and therefore $\Lambda $ is $%
\mathcal{UR}$. This is what the next result shows.

\begin{corollary}
\label{RS1}Let $\Lambda =\{\lambda _{1},\lambda _{2},\ldots ,\lambda _{n}\},$
$\lambda _{i}\neq 0,$ $i=2,\ldots ,n,$ $s_{1}(\Lambda )>0,$ be a realizable
left half-plane list of complex numbers. If there is a decomposition of $%
\Lambda $ as in Corollary \ref{Main}, with all lists $\Gamma _{k}$ being the
spectrum of a diagonalizable nonnegative companion matrix $A_{k},$ with a
positive column, then $\Lambda $ is universally realizable.
\end{corollary}

\begin{proof}
It is enough to prove the result for two lists $\Gamma _{k}$ of the
decomposition of $\Lambda .$ Let $\Gamma _{k-1}$ and $\Gamma _{k},$ $%
k=2,\ldots ,t,$ be the spectrum, respectively, of matrices $A_{k-1}$ and $%
A_{k},$ which are diagonalizable nonnegative companion with a positive
column (the last one). Then $A_{k-1}$ and $A_{k}$ are irreducible. In
particular, $A_{k}$ has a positive eigenvector $\mathbf{s}$ and, since $%
A_{k}^{T}$ is also irreducible, $A_{k}$ has also a positive left eigenvector 
$\mathbf{t}^{T}$ with $\mathbf{t}^{T}\mathbf{s}=1.$ Now, let 
\begin{equation*}
A_{k-1}=\left[ 
\begin{array}{cc}
A_{1,k-1} & \mathbf{a} \\ 
\mathbf{b}^{T} & s_{1}(\Gamma _{k-1})%
\end{array}%
\right] .
\end{equation*}%
Since the last column of $A_{k-1}$ is positive, the vector $\mathbf{a}$ is
also positive and $\mathbf{at}^{T}$ is a positive submatrix. Therefore, the
glue of $A_{k-1}$ with $A_{k},$ 
\begin{equation*}
C_{k}=\left[ 
\begin{array}{cc}
A_{1,k-1} & \mathbf{at}^{T} \\ 
\mathbf{sb}^{T} & A_{k}%
\end{array}%
\right] ,
\end{equation*}%
is a diagonalizable nonnegative matrix with its last column being positive.
Note that $C_{k}$ is also irreducible. Then $C_{k}$ has, besides, a positive
eigenvector, and from Lemma \ref{CS} $C_{k}$ is similar to a matrix with
constant row sums and with its last column being positive. Thus, $C_{k}$ is
similar to a diagonalizable positive matrix. Then, \v{S}migoc's glue applied
to all matrices $A_{k}$ gives an $n$-by-$n$ diagonalizable irreducible
nonnegative matrix $A$ with a positive column and spectrum $\Lambda $.
Therefore, $A$ is similar to a diagonalizable positive matrix with spectrum $%
\Lambda $ and from the extension in \cite{Collao} $\Lambda $ is $\mathcal{UR}
$.
\end{proof}

\bigskip

Observe that if $\lambda _{i}\neq 0,$ $i=2,\ldots ,n;$ $s_{1}(\Lambda )>0;$ $%
b_{2}(A_{k})>0$ in Corollary \ref{RS1}, then we can guarantee the existence
of an $n$-by-$n$ diagonalizable nonnegative irreducible matrix $A$ with
spectrum $\Lambda $ and a positive column. Thus, this is enough to show the
universal realizability of $\Lambda .$

\begin{example}
Consider the list%
\begin{eqnarray*}
\Lambda &=&\{23,-2,-2,-1\pm 5i,-1\pm 5i,-1\pm 5i,-2\pm 7i,-2\pm 7i\},\text{
with} \\
\Gamma _{1} &=&\{23,-1\pm 5i\},\text{ }\Gamma _{2}=\{21,-2,-1\pm 5i,-2\pm
7i\}, \\
\Gamma _{3} &=&\{13,-2,-1\pm 5i,-2\pm 7i\}.
\end{eqnarray*}%
The diagonalizable companion matrices 
\begin{eqnarray*}
A_{1} &=&\left[ 
\begin{array}{ccc}
0 & 0 & 598 \\ 
1 & 0 & 20 \\ 
0 & 1 & 21%
\end{array}%
\right] ,\text{ }A_{2}=\left[ 
\begin{array}{cccccc}
0 & 0 & 0 & 0 & 0 & 57\,876 \\ 
1 & 0 & 0 & 0 & 0 & 35\,002 \\ 
0 & 1 & 0 & 0 & 0 & 6266 \\ 
0 & 0 & 1 & 0 & 0 & 1695 \\ 
0 & 0 & 0 & 1 & 0 & 69 \\ 
0 & 0 & 0 & 0 & 1 & 13%
\end{array}%
\right] ,\text{ } \\
A_{3} &=&\left[ 
\begin{array}{cccccc}
0 & 0 & 0 & 0 & 0 & 35\,828 \\ 
1 & 0 & 0 & 0 & 0 & 20\,618 \\ 
0 & 1 & 0 & 0 & 0 & 3194 \\ 
0 & 0 & 1 & 0 & 0 & 903 \\ 
0 & 0 & 0 & 1 & 0 & 5 \\ 
0 & 0 & 0 & 0 & 1 & 5%
\end{array}%
\right]
\end{eqnarray*}%
realize lists $\Gamma _{1},\Gamma _{2}$ and $\Gamma _{3},$ respectively. 
\v{S}migoc's glue technique applied to matrices $A_{1},A_{2}$ and $A_{3}$
gives a $13$-by-$13$ diagonalizable irreducible nonnegative matrix with a
positive column and spectrum $\Lambda .$ Therefore, from Lemma \ref{CS} and 
\cite{Collao}, $\Lambda $ is UR.
\end{example}

\section{The effect of adding a negative real number to a not UR list}

\noindent In this section we show how to add a negative real number $-c$ to
a list of complex numbers 
\begin{equation*}
\Lambda =\{\lambda ,-a\pm bi,\ldots ,-a\pm bi\},\text{ }\lambda ,a,b>0,\text{
with\ }s_{1}(\Lambda )>0,
\end{equation*}%
which is not $\mathcal{UR}$ or we do not know whether it is, makes%
\begin{equation*}
\Lambda _{c}=\{\lambda ,-c,\underset{(n-2)\text{ complex numbers}}{%
\underbrace{-a\pm bi,\ldots ,-a\pm bi}}\}
\end{equation*}%
$\mathcal{UR}$. For instance, the list $\Lambda _{1}=\{6,-1\pm 3i,-1\pm 3i\}$
is realizable, but we do not know whether it is $\mathcal{UR}$, while $%
\Lambda _{2}=\{17,-3\pm 9i,-3\pm 9i\}$ is not realizable. However, both
lists become $\mathcal{UR}$ if we add an appropriate negative real number $%
-c $ to each of them.

We start this section with a lemma which gives a formula to compute the
coefficient $b_{2}$ in (\ref{b2}), Lemma \ref{LS2}, for lists $\Lambda _{c}$

\begin{lemma}
\label{JA}Let 
\begin{equation*}
\Lambda _{c}=\{\lambda ,-c,\underset{(n-2)\text{ complex numbers}}{%
\underbrace{-a\pm bi,\ldots ,-a\pm bi}}\}
\end{equation*}%
be a realizable left half-plane lists of complex numbers and let $\Lambda
_{c}=\Lambda _{1}\cup \Lambda _{2}\cup \cdots \cup \Lambda _{t}$ be a
decomposition of $\Lambda _{c},$ $-c\in \Lambda _{t},$ with auxiliary lists $%
\Gamma _{k}$ with realizing companion matrices $A_{k},$ $k=1,2,\ldots ,t,$
as in Corollary \ref{Main}, associated with a desired JCF allowed by $%
\Lambda _{c}.$ Then the entry in position $(n-1,n)$ of a matrix $A_{k},$ $%
k=1,2,\ldots ,t,$ is 
\begin{equation}
b_{2}=p(2a\lambda -2a^{2}n+(4k-2p+1)a^{2}-b^{2})+c(\lambda -(n-2)a),
\label{Fb2}
\end{equation}%
where $(k-1)$ is the number of pairs $-a\pm bi$ of the last list $\Gamma
_{t} $ of the diagonalizable decomposition of $\Lambda _{c},$ plus the
number of pairs $-a\pm bi$ of each previous list $\Gamma _{k},$ $k=1,\ldots
,t-1,$ of the decomposition, and $p$ is the number of pairs $-a\pm bi$ of
the corresponding list $\Gamma _{k}.$ Moreover, $b_{2}$ increases if $k$
increases.
\end{lemma}

\begin{proof}
It is well known that $b_{2}=\dsum\limits_{1\leq j_{1}<j_{2}\leq n}\lambda
_{j_{1}}\lambda _{j_{2}},$ with $\lambda _{ji}\in $ $\Gamma _{k},$ from
which $b_{2}$ in (\ref{Fb2}) is obtained. Moreover it is clear that $b_{2}$
increases when $k$ increases.
\end{proof}

\begin{example}
Consider%
\begin{equation*}
\Lambda _{c}={\Huge \{}\frac{77}{4},-3,\underset{8\text{ complex numbers}}{%
\underbrace{-2\pm 5i,\ldots ,-2\pm 5i}}{\Huge \}}.\text{ }
\end{equation*}%
The last diagonalizable list from the diagonalizable decomposition of $%
\Lambda _{c}$ is 
\begin{equation*}
\Gamma _{4}:(x-\frac{29}{4})(x+3)(x+2-5i)(x+2+5i)
\end{equation*}%
with realizing matrix$\ $\newline
\begin{eqnarray*}
A_{4} &=&\left[ 
\begin{array}{cccc}
0 & 0 & 0 & \frac{2523}{4} \\ 
1 & 0 & 0 & \frac{841}{4} \\ 
0 & 1 & 0 & \frac{\mathbf{39}}{\mathbf{4}} \\ 
0 & 0 & 1 & \frac{1}{4}%
\end{array}%
\right] \longrightarrow b_{2}(A_{4})=\frac{39}{4} \\
b_{2} &=&p(2a\lambda -2a^{2}n+(4k-2p+1)a^{2}-b^{2})+c(\lambda -(n-2)a) \\
b_{2}(A_{4}) &=&(4\frac{{\small 77}}{{\small 4}}-80+(8-2+1)4-25)+3(\frac{%
{\small 77}}{{\small 4}}-16)=\frac{39}{4}.
\end{eqnarray*}%
Suppose we want to obtain a nonnegative matrix with JCF%
\begin{equation*}
\mathbf{J}=diag\{J_{1}(\frac{77}{4}),J_{1}(-3),J_{2}(-2+5i),(J_{2}(-2-5i)\}.
\end{equation*}%
Then,%
\begin{eqnarray*}
\Gamma _{1}^{\prime } &=&\{\frac{77}{4},-2\pm 5i,-2\pm 5i\} \\
\Gamma _{2}^{\prime } &=&\{\frac{45}{4},-3,-2\pm 5i,-2\pm 5i\}.
\end{eqnarray*}%
If $A_{1}^{\prime },A_{2}^{\prime }$ are companion realizing matrices for $%
\Gamma _{1}^{\prime }$ and $\Gamma _{2}^{\prime },$ respectively, then from
Lemma \ref{JA}, $b_{2}(A_{2}^{\prime })=\frac{103}{4},$ $b_{2}(A_{1}^{\prime
})=80$ guarantee that $A_{1}^{\prime }$ and $A_{2}^{\prime }$ are
nonnegative. Next, the glue of $A_{1}^{\prime }$ with $A_{2}^{\prime }$
gives a nonnegative matrix with JCF $\mathbf{J.}$
\end{example}

\begin{theorem}
\label{negc}Let $\Lambda =\{\lambda ,-a\pm bi,\ldots ,-a\pm bi\},$ fixed $%
\lambda ,$ $a,b>0,$ be a list of complex numbers with $s_{1}(\Lambda )>0.$
If 
\begin{equation}
\frac{(2n-11)a^{2}+b^{2}}{2a}\leq \lambda ,  \label{C1}
\end{equation}%
and there is a real number $c>0$ such that%
\begin{equation}
\frac{2a(na-\lambda )+b^{2}-7a^{2}}{\lambda -(n-2)a}\leq c\leq \lambda
-(n-2)a,  \label{C2}
\end{equation}%
then 
\begin{equation*}
\Lambda _{c}=\{\lambda ,-c,\underset{(n-2)\text{ complex numbers}}{%
\underbrace{-a\pm bi,\ldots ,-a\pm bi}}\}
\end{equation*}%
becomes universally realizable.
\end{theorem}

\begin{proof}
Consider the decomposition $\Lambda _{c}=\Lambda _{1}\cup \Lambda _{2}\cup
\cdots \cup \Lambda _{\frac{n-2}{2}},$ with%
\begin{eqnarray*}
\Lambda _{1} &=&\{\lambda ,-a\pm bi\},\text{ } \\
\Lambda _{k} &=&\{-a\pm bi\},\text{ }k={\small 2,\ldots ,}\text{$\frac{n-4}{2%
}$},\text{ } \\
\Lambda _{\frac{n-2}{2}} &=&\{-c,-a\pm bi\}.
\end{eqnarray*}%
We take the auxiliary sub-lists%
\begin{eqnarray*}
\Gamma _{1} &=&\Lambda _{1}=\{\lambda ,-a\pm bi\} \\
\Gamma _{2} &=&\{\lambda -2a,-a\pm bi\} \\
\Gamma _{3} &=&\{\lambda -4a,-a\pm bi\} \\
&&\vdots \\
\Gamma _{\frac{n-4}{2}} &=&\{\lambda -(n-6)a,-a\pm bi\}, \\
\Gamma _{\frac{n-2}{2}}\ &=&\{\lambda -(n-4)a,-c,-a\pm bi\},
\end{eqnarray*}%
where $\Gamma _{\frac{n-4}{2}}$ and $\Gamma _{\frac{n-2}{2}}$ are the
spectrum of the diagonalizable companion matrices%
\begin{equation*}
A_{\frac{n-4}{2}}=\left[ 
\begin{array}{ccc}
0 & 0 & (a^{2}+b^{2})(\lambda -(n-6)a) \\ 
1 & 0 & 2a\lambda -a^{2}(2n-11)-b^{2} \\ 
0 & 1 & \lambda -(n-4)a%
\end{array}%
\right]
\end{equation*}%
and 
\begin{equation*}
A_{\frac{n-2}{2}}=\left[ 
\begin{array}{cccc}
0 & 0 & 0 & (a^{2}+b^{2})(\lambda -(n-4)a)c \\ 
1 & 0 & 0 & (a^{2}+b^{2})(\lambda -(n-4)a)+(7a^{2}-b^{2}+2a\lambda -2a^{2}n)c
\\ 
0 & 1 & 0 & (\lambda -(n-2)a)c+(7a^{2}-b^{2}+2a\lambda -2a^{2}n) \\ 
0 & 0 & 1 & \lambda -(n-2)a-c%
\end{array}%
\right] ,
\end{equation*}%
respectively. Observe that sub-lists $\Gamma _{\frac{n-6}{2}},\ldots ,\Gamma
_{2},\Gamma _{1}$ have the same pair of complex numbers that the list $%
\Gamma _{\frac{n-4}{2}},$ but with a bigger Perron eigenvalue. Then, if $%
\Gamma _{\frac{n-4}{2}}$ is diagonalizably companion realizable, $\Gamma _{%
\frac{n-6}{2}},\ldots ,\Gamma _{2},\Gamma _{1}$ also are. Thus, from Lemma %
\ref{LS2} we only need to consider the entries in position $(2,3)$ in $A_{%
\frac{n-4}{2}}$ and in position $(3,4)$ in $A_{\frac{n-2}{2}}.$ From (\ref%
{C1}) and (\ref{C2}) these entries are nonnegative and therefore $A_{\frac{%
n-4}{2}}$ and $A_{\frac{n-2}{2}}$ are diagonalizable companion realizing
matrices. Thus, after applying $\frac{n-4}{2}$ times \v{S}migoc's glue to
the matrices $A_{1},\ldots ,A_{\frac{n-2}{2}}$, we obtain an $n$-by-$n$
diagonalizable nonnegative matrix $A$ with spectrum $\Lambda _{c}.$ Thus $%
\Lambda _{c}$ is $\mathcal{DR}.$\newline
To obtain an $n$-by-$n$ nonnegative matrix $A$ with spectrum $\Lambda _{c}$
and a non-diagonal JCF $\mathbf{J}$, we take $\Lambda _{c}=\Lambda _{1}\cup
\cdots \cup \Lambda _{t}$ with auxiliary lists $\Gamma _{k}$ being the
spectrum of a companion matrix $A_{k}$ with JCF as a sub-matrix of $\mathbf{%
J.}$ Next we need to prove that all $A_{k}$ are nonnegative. To do that, we
compute $b_{2}(A_{t})$ from the formula in (\ref{Fb2}), where $A_{t}$ (with $%
\Gamma _{t}$ containing $-c)$ is the last diagonalizable matrix in the
diagonalizable decomposition of $\Lambda _{c}.$ From (\ref{C1}) and (\ref{C2}%
) $b_{2}(A_{t})\geq 0.$ From Lemma \ref{JA} all $b_{2}(A_{k}),$ $k=1,\ldots
,t-1,$ are nonnegative. Therefore the glue of matrices $A_{k}$ gives an $n$%
-by-$n$ nonnegative matrix $A$ with the desired JCF $\mathbf{J.}$
\end{proof}

\begin{example}
$i)$ $\Lambda =\{6,-1\pm 3i,-1\pm 3i\}$ is realizable by the companion matrix%
\begin{equation*}
C=\left[ 
\begin{array}{ccccc}
0 & 0 & 0 & 0 & 600 \\ 
1 & 0 & 0 & 0 & 140 \\ 
0 & 1 & 0 & 0 & 104 \\ 
0 & 0 & 1 & 0 & 0 \\ 
0 & 0 & 0 & 1 & 2%
\end{array}%
\right] ,
\end{equation*}%
with a non-diagonal JCF. We do not know whether $\Lambda $ has a
diagonalizable realization. Then, consider the list 
\begin{equation*}
\Lambda _{c}=\{6,-c,-1\pm 3i,-1\pm 3i\}.
\end{equation*}%
Condition (\ref{C1}) is satisfied and from (\ref{C2}) we have $1\leq c\leq
2. $ Then for $c=2,$ we have that 
\begin{equation*}
\Gamma _{1}=\{6,-1\pm 3i\},\text{ \ }\Gamma _{2}=\{4,-2,-1\pm 3i\}
\end{equation*}%
are the spectrum of diagonalizable nonnegative companion matrices 
\begin{equation*}
A_{1}=\left[ 
\begin{array}{ccc}
0 & 0 & 60 \\ 
1 & 0 & 2 \\ 
0 & 1 & 4%
\end{array}%
\right] ,\text{ and \ }A_{2}=\left[ 
\begin{array}{cccc}
0 & 0 & 0 & 80 \\ 
1 & 0 & 0 & 36 \\ 
0 & 1 & 0 & 2 \\ 
0 & 0 & 1 & 0%
\end{array}%
\right] ,
\end{equation*}%
respectively. Then, from \v{S}migoc's glue we obtain a diagonalizable
nonnegative matrix with spectrum $\Lambda _{c}.$ It is clear that, from the
characteristic polynomial associated to $\Lambda _{c},$ $\Lambda _{c}$ has
also a companion realization $A_{3},$ 
\begin{equation*}
A_{3}=\left[ 
\begin{array}{cccccc}
0 & 0 & 0 & 0 & 0 & 1200 \\ 
1 & 0 & 0 & 0 & 0 & 880 \\ 
0 & 1 & 0 & 0 & 0 & 348 \\ 
0 & 0 & 1 & 0 & 0 & 104 \\ 
0 & 0 & 0 & 1 & 0 & 4 \\ 
0 & 0 & 0 & 0 & 1 & 0%
\end{array}%
\right] ,
\end{equation*}%
with a JCF with blocks of maximum size. Note that the formula in (\ref{Fb2})
gives $(k=2,$ $p=1,$ $t=2)$ $b_{2}(A_{2})=2,$ while $(k=3,$ $p=2)$ gives $%
b_{2}(A_{3})=4.$ Therefore $\Lambda _{c}$ is $\mathcal{UR}$. Observe that if 
$1\leq $ $c\leq 2,$ then 
\begin{equation*}
\Lambda _{c}=\{6,-c,-1\pm 3i,-1\pm 3i\}
\end{equation*}%
is also $\mathcal{UR}$.\newline
$ii)$ Consider the list $\Lambda =\{17,-3\pm 9i,-3\pm 9i\}.$ Since $%
s_{1}(\Lambda )=5$ and $s_{2}(\Lambda )=1,$ $\Lambda $ is not realizable.
From condition (\ref{C2}), $\frac{24}{5}\leq c\leq 5.$ Then for $c=5,$ 
\begin{equation*}
\Lambda _{c}=\{17,-5,-3\pm 9i,-3\pm 9i\}\text{ }
\end{equation*}%
is $\mathcal{UR}$. In fact, 
\begin{equation*}
\Gamma _{1}=\{17,-3\pm 9i\}\text{ and }\Gamma _{2}=\{11,-5,-3\pm 9i\}
\end{equation*}%
are the spectrum of diagonalizable nonnegative companion matrices, which
from \v{S}migoc's glue give rise to a diagonalizable nonnegative matrix with
spectrum $\Lambda _{c}.$ From the characteristic polynomial associated to $%
\Lambda _{c}$ we obtain a nonnegative companion matrix with spectrum $%
\Lambda _{c}$ and non-diagonal JCF. Therefore, $\Lambda _{c}$ is $\mathcal{UR%
}.$
\end{example}

Observe that in Theorem \ref{negc}, in spite that $s_{1}(\Lambda )>0,$ if $%
s_{1}(\Lambda )$ is small enough, there are lists $\Lambda _{c},$ which are
not $\mathcal{UR}$ or we cannot to prove they are from our procedure.
However, from Theorem \ref{negc} we may compute a Perron eigenvalue $\lambda
,$ which guarantees that for a family of lists $\Lambda _{c},$ with $c>0$
and $n\geq 6,$ $\Lambda _{c}$ will be $\mathcal{UR}$. Then, the following
result characterizes a family of left half-plane lists, which are $\mathcal{%
UR}$.

\begin{corollary}
\label{corR}The left half-plane lists of the family 
\begin{equation*}
\Lambda _{c}={\Huge \{}\frac{1}{2a}((2n-7)a^{2}+b^{2}),-c,\underset{(n-2)%
\text{ complex numbers}}{\underbrace{-a\pm bi,\ldots ,-a\pm bi}}{\Huge \}},
\end{equation*}%
with $0<\sqrt{3}a<b,$ $0<c\leq \frac{b^{2}-3a^{2}}{2a},$ are universally
realizable..
\end{corollary}

\begin{proof}
It is clear that for $\lambda =\frac{1}{2a}\left( (2n-7)a^{2}+b^{2}\right) ,$
conditions (\ref{C1}) and (\ref{C2}) in Theorem \ref{negc} are satisfied.
Moreover, from $0<\sqrt{3}a<b,$ $\lambda -(n-2)a=\frac{b^{2}-3a^{2}}{2a}>0.$
\end{proof}

\bigskip

Then, from Corollary \ref{corR} some left half-plane lists that are UR are:%
\begin{eqnarray*}
i)\ \Lambda _{c} &=&{\Huge \{}\frac{2n-3}{2}a,-c,\underset{(n-2)\text{
complex numbers}}{\underbrace{-a\pm 2ai,\ldots ,-a\pm 2ai}}{\Huge \}},\text{
with }0<c\leq \frac{a}{2} \\
&& \\
ii)\ \Lambda _{c} &=&{\Huge \{}(n+1)a,-c,\underset{(n-2)\text{ complex
numbers}}{\underbrace{-a\pm 3ai,\ldots ,-a\pm 3ai}}{\Huge \}},\text{ with }%
0<c\leq 3a \\
&&. \\
iii)\text{ }\Lambda _{c} &=&{\Huge \{}\frac{2n+9}{2}a,-c,\underset{(n-2)%
\text{ complex numbers}}{\underbrace{-a\pm 4ai,\ldots ,-a\pm 4ai}}{\Huge \}},%
\text{ with }0<c\leq \frac{13}{2}a \\
&& \\
iv)\text{ }\Lambda _{c} &=&{\Huge \{}\frac{8n-3}{8}a,-c,\underset{(n-2)\text{
complex numbers}}{\underbrace{-a\pm \frac{5}{2}ai,\ldots ,-a\pm \frac{5}{2}ai%
}}{\Huge \}},\text{ with }0<c\leq \frac{13}{8}a,
\end{eqnarray*}%
and so on.

\bigskip

Observe that in Corollary \ref{corR}, if $c$ is strictly less than its upper
bound, then $\Lambda _{c},$ as we have seen, can be realized by a
diagonalizable matrix with its last column being positive. Then, from the
extension in \cite{Collao}, $\Lambda _{c}$ is $\mathcal{UR}$.

\section{The merge of spectra}

Let $\Gamma _{1}=\{\lambda _{1},\lambda _{2},\ldots ,\lambda _{n}\}$ and $%
\Gamma _{2}=\{\mu _{1},\mu _{2},\ldots ,\mu _{m}\}$ be lists of complex
numbers. In \cite{Johnson5} the authors define the concept of the \textit{%
merge of the spectra} $\Gamma _{1}$ with\textit{\ }$\Gamma _{2}$ as%
\begin{equation*}
\Gamma =\{\lambda _{1}+\mu _{1},\lambda _{2},\ldots ,\lambda _{n},\mu
_{2},\ldots ,\mu _{m}\},
\end{equation*}%
and prove that if $\Gamma _{1}$ and\textit{\ }$\Gamma _{2}$ are
diagonalizably ODP realizable, then the merge $\Gamma _{1}$\textit{\ with }$%
\Gamma _{2},$ is also diagonalizably ODP realizable, and therefore from the
extension in \cite{Johnson4}, $\Gamma $ is $\mathcal{UR}$. Here we set a
similar result as follows:

\begin{theorem}
\label{merge}Let $\Gamma _{1}=\{\lambda _{1},\lambda _{2},\ldots ,\lambda
_{n}\},$ $\lambda _{1}>\left\vert \lambda _{i}\right\vert ,$ $i=2,\ldots ,n,$
be the spectrum of a diagonalizable nonnegative $n$-by-$n$ matrix $A\in 
\mathcal{CS}_{\lambda _{1}}$ with its last column being positive. Let $%
\Gamma _{2}=\{\mu _{1},\mu _{2},\ldots ,\mu _{m}\},$ $\mu _{1}>\left\vert
\mu _{i}\right\vert ,$ $i=2,\ldots ,m,$ be the spectrum of a diagonalizable
nonnegative $m$-by-$m$ matrix $B\in \mathcal{CS}_{\mu _{1}}$ with its last
column being positive. Then 
\begin{equation*}
\Gamma =\{\lambda _{1}+\mu _{1},\lambda _{2},\ldots ,\lambda _{n},\mu
_{2},\ldots ,\mu _{m}\}
\end{equation*}%
is universally realizable..
\end{theorem}

\begin{proof}
Let $A\in \mathcal{CS}_{\lambda _{1}}$ be a diagonalizable nonnegative
matrix with spectrum $\Gamma _{1}$ and with its last column being positive.
Then $A$ is similar to a diagonalizable positive matrix $A^{\prime }.$ If $%
\alpha _{1},\ldots ,\alpha _{n}$ are the diagonal entires of $A^{\prime },$
then%
\begin{equation*}
A_{1}=A^{\prime }+\mathbf{e}[0,0,\ldots ,\mu _{1}]=\left[ 
\begin{array}{cc}
A_{11}^{\prime } & \mathbf{a} \\ 
\mathbf{b}^{T} & \alpha _{n}+\mu _{1}%
\end{array}%
\right] \in \mathcal{CS}_{\lambda _{1}+\mu _{1}}
\end{equation*}%
is diagonalizable positive with spectrum $\{\lambda _{1}+\mu _{1},\lambda
_{2},\ldots ,\lambda _{n}\}$ and diagonal entries $\alpha _{1},\alpha
_{2},\ldots ,\alpha _{n}+\mu _{1}.$ Let $B\in \mathcal{CS}_{\mu _{1}}$ be a
diagonalizable nonnegative matrix with spectrum $\Gamma _{2}$ and with its
last column being positive. Then $B$ is similar to a diagonalizable positive
matrix $B^{\prime }$ and%
\begin{equation*}
B_{1}=B^{\prime }+\mathbf{e}[\alpha _{n},0,\ldots ,0]
\end{equation*}%
is diagonalizable positive with spectrum $\{\mu _{1}+\alpha _{n},\mu
_{2},\ldots ,\mu _{m}\}.$ Now, by applying the \v{S}migoc's glue to matrices 
$A_{1}$ and $B_{1}$, we obtain a diagonalizable positive matrix $C$ with
spectrum $\Gamma .$ Hence, $\Gamma $ is $\mathcal{UR}$
\end{proof}

\bigskip

Theorem \ref{merge} is useful to decide, in many cases, about the universal
realizability of left half-plane list of complex numbers, as for instance:

\begin{example}
Is the list 
\begin{equation*}
\Gamma ={\LARGE \{}30\text{$,-1$},-5,-1\pm 3i,-1\pm 3i,-1\pm 3i,-3\pm
9i,-3\pm 9i{\LARGE \}}\text{ }\mathcal{UR}\text{?}
\end{equation*}%
Observe that from the results in Section $4,$ 
\begin{eqnarray*}
\Gamma _{1} &=&\{21,-5,-3\pm 9i,-3\pm 9i\}. \\
\Gamma _{2} &=&\{9,-1,-1\pm 3i,-1\pm 3i,-1\pm 3i\}
\end{eqnarray*}%
are the spectrum of a diagonalizably nonnegative matrix with constant row
sums and a positive column (the last one). Then, they are similar to
diagonalizable positive matrices and from Theorem \ref{merge}, the merge $%
\Gamma $ is also the spectrum of a diagonalizable positive matrix.
Therefore, $\Gamma $ is $\mathcal{UR}$.
\end{example}

\end{document}